\documentclass{article}

\usepackage[final]{neurips_2024}

\usepackage[utf8]{inputenc} % allow utf-8 input
\usepackage[T1]{fontenc}    % use 8-bit T1 fonts
\usepackage{hyperref,graphicx}       % hyperlinks
\usepackage{url,amsthm,amsmath}            % simple URL typesetting
\usepackage{booktabs}       % professional-quality tables
\usepackage{amsfonts}       % blackboard math symbols
\usepackage{nicefrac}       % compact symbols for 1/2, etc.
\usepackage{microtype}      % microtypography
\usepackage{xcolor}         % colors
\setcitestyle{aysep={,}}

\newcommand\abs[1]{\left\lvert#1\right\rvert}

\newcommand{\E}{\mathbb{E}}

\newcommand{\X}{\mathcal{X}}

\newcommand{\gti}{\rightarrow\infty}
\newcommand{\gtz}{\rightarrow0}

\newtheorem*{theorem*}{Theorem}

\theoremstyle{definition}

\theoremstyle{remark}

\title{Rubik’s Cube Scrambling Requires at Least 26 Random Moves}

\renewcommand*{\thefootnote}{\fnsymbol{footnote}}
\author{%
  Yanlin Qu$^*$ \quad Tomas Rokicki \quad Hillary Yang\\
  $^*$Stanford University\\
  \texttt{quyanlin@alumni.stanford.edu}\footnotemark
}

\begin{document}

\maketitle

\begin{abstract}
Scrambling the standard 3x3x3 Rubik’s Cube corresponds to a random walk on a group containing approximately 43 quintillion elements. Viewing the random walk as a Markov chain, its mixing time determines the number of random moves required to sufficiently scramble a solved cube. With the aid of a supercomputer, we show that the mixing time is at least 26, providing the first non-trivial bound.
\end{abstract}

\section{Introduction}
The Rubik’s Cube, invented by Ern\H{o} Rubik in 1974, is a classic combinatorial puzzle with approximately 43 quintillion possible configurations. A natural question to ask is:
\begin{quote}
How many moves are required to solve any configuration of the cube?
\end{quote}
This question has been definitely answered: any configuration can be solved in 20 moves or fewer; see \cite{rokicki2014diameter}.
However, the opposite question remains open:
\begin{quote}
How many moves are required to sufficiently scramble a solved cube?
\end{quote}
If the answer to the first question is known as God’s number, then the answer to the second question could be called the {\it Ragnar\"{o}k countdown.}
When the cube is thoroughly scrambled, each configuration is equally likely to occur (i.e., uniform distribution). Figuring out the minimum number of random moves required to approximately reach the uniform distribution is of both theoretical and practical interest, especially in the speedcubing community. A similar problem in the context of card shuffling has been analytically addressed: seven riffle shuffles are enough to approximately mix a standard deck of playing cards; see \cite{bayer1992trailing}. Given the Rubik’s Cube’s far more complex structure compared to a deck of cards, computational methods, rather than analytical ones, are required to study the scrambling process.

In this paper, with the aid of a supercomputer, we show that at least 26 random moves are required to sufficiently scramble a solved cube. In other words, 25 random moves are insufficient to ensure that each configuration occurs with approximately equal probability. Formally, repeatedly applying independent and identically distributed (iid) random moves to the cube generates a Markov chain that converges to the uniform distribution. By convention, the mixing time is defined as the time when the total variation (TV) distance towards the limiting distribution drops below one-quarter. Our conclusion is that the mixing time is no less than 26, which is the first non-trivial bound.

The rest of the paper is organized as follows: In Section \ref{preliminaries}, we introduce the notation for describing the Rubik’s Cube, review the concept of Markov chain mixing time, and provide an overview of the supercomputer used in this study. In Section \ref{main}, we prove the lower bound. In Section \ref{tri}, we present a supercomputer-generated dataset that enables a more in-depth convergence analysis.

\footnotetext{Current affiliation: Columbia Business School. Working email: \texttt{qu.yanlin@columbia.edu}.}

\renewcommand*{\thefootnote}{\arabic{footnote}}
\setcounter{footnote}{0}
\section{Preliminaries}
\label{preliminaries}

\begin{figure*}[ht]
\begin{center}
\centerline{\includegraphics[width=0.3\columnwidth]{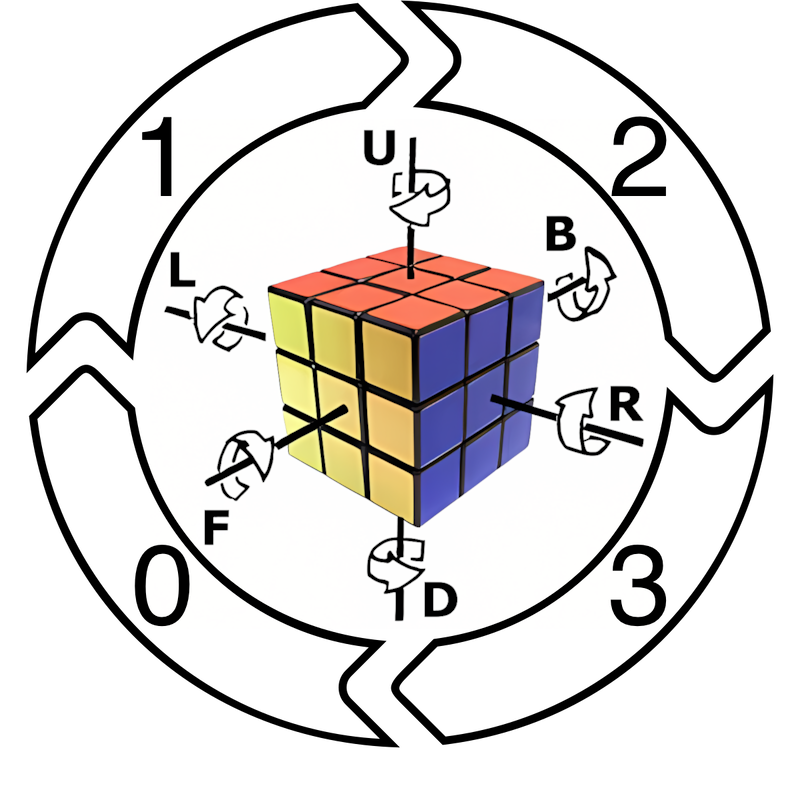}}
\caption{U: up, D: down, F: front, B: back, L: left, R: right; 1: 90$^\circ$, 2: 180$^\circ$, 3: 270$^\circ$} (clockwise); e.g., R3: rotating the right face 270 degrees clockwise.\footnotemark
\label{cube}
\end{center}
\end{figure*}
\footnotetext{Image source: \url{https://rubiks.fandom.com/wiki/Fridrich_Method}, modified by the first author.}
\subsection{Rubik's Cube notation}
As shown in Figure \ref{cube}, the six faces of the Rubik’s Cube are labeled with the letters U (up), D (down), F (front), B (back), L (left), and R (right). Each face can be rotated 90, 180, or 270 degrees clockwise, labeled as 1, 2, and 3, respectively. (Focusing on clockwise rotations makes the notation more systematic.) The six faces and three rotation options result in a total of 18 distinct moves, namely:
\[
\mathrm{U1, U2, U3, D1, D2, D3, F1, F2, F3, B1, B2, B3, L1, L2, L3, R1, R2, R3}.
\]
For example, R3 means rotating the right face 270 degrees clockwise.
\subsection{Markov chain mixing time}
Let $\X$ denote the space of all possible configurations. Let $f$ be a move sampled uniformly at random from the 18 distinct moves described above, which randomly changes the Rubik’s Cube from one configuration to another. Let $f_1,f_2,...$ be iid copies of $f$. The scrambling process is described by the recursion
\[
X_{n+1}=f_{n+1}(X_n),\;\;n=0,1,2,...,
\]
which defines a Markov chain $X=(X_n:n\geq0)$ on $\X$. This Markov chain is irreducible because it can reach any state (configuration) in a finite number of steps (moves). It is also aperiodic, as it can return to a state in either two steps (e.g., R1R3) or three steps (e.g., R1R2R1). An irreducible and aperiodic Markov chain on a finite state space converges to a unique stationary distribution, regardless of its initial state; see, e.g., \cite{levin2017markov}. The unique stationary distribution, denoted by $X_\infty$, must be the uniform distribution due to the detailed balance equations. As $n\gti$, $X_n$ converges to $X_\infty$ in TV distance
\[
\mathrm{TV}(X_n,X_\infty)=\sup_{h:|h|\leq1}\abs{\E h(X_n)-\E h(X_\infty)}=(1/2)\sum_{x\in\X}\abs{P(X_n=x)-P(X_\infty=x)}\gtz,
\]
where all probabilities ($P$) and expectations ($\E$) are conditional on starting from the solved state. The mixing time is defined as the smallest $n$ such that the above TV distance is small
\[
t_{\mathrm{mix}}(\epsilon)=\min\{n:\mathrm{TV}(X_n,X_\infty)\leq \epsilon\}.
\]
Although the convention is to set $t_{\mathrm{mix}}=t_{\mathrm{mix}}(1/4)$, to better quantify the convergence, we bound each $t_{\mathrm{mix}}(\epsilon)$ from below by bounding each $\mathrm{TV}(X_n,X_\infty)$ from below.

\subsection{Big Red 200}

Indiana University’s supercomputer Big Red 200 is an HPE Cray EX supercomputer featuring 640 compute nodes, each equipped with 256 GB of memory and two 64-core, 2.25 GHz, 225-watt AMD EPYC 7742 processors.  Big Red 200 also includes 64 GPU-accelerated nodes, each with 256 GB of memory, a single 64-core, 2.0 GHz, 225-watt AMD EPYC 7713 processor, and four NVIDIA A100 GPUs. Big Red 200 has a theoretical peak performance (Rpeak) of nearly 7 petaFLOPS.
\section{Main result}
\begin{theorem*}
    Rubik’s Cube scrambling requires at least 26 random moves, i.e., $t_{\mathrm{mix}}\geq26.$
\end{theorem*}
\begin{proof}
    Let $m=25$. It suffices to show that $\mathrm{TV}(X_m,X_\infty)>0.25.$ However, due to the astronomical size of $\X$, estimating $\mathrm{TV}(X_m,X_\infty)$ via Monte Carlo is infeasible, so we need to find a computable lower bound for it. In what follows, we refer to the solved state as the origin, denoted by $o$. For $x\in\X$, let $d_o(x)\in\{0,...,20\}$ be the distance between $x$ and $o$, i.e., the minimum number of moves required to solve configuration $x$ is $d_o(x)$. With $Y=d_o(X)$, we have
    \begin{align*}
        \mathrm{TV}(Y_m,Y_\infty)=&\sup_{l:|l|\leq1}\abs{\E l(Y_m)-\E l(Y_\infty)}\\
        =&\sup_{l:|l|\leq1}\abs{\E l(d_o(X_m))-\E l(d_o(X_\infty))}\\
        \leq&\sup_{h:|h|\leq1}\abs{\E h(X_m)-\E h(X_\infty)}\\
        =&\mathrm{TV}(X_m,X_\infty)
    \end{align*}
    where the inequality is because $|l\circ d_o|\leq1$. From $X$ to $Y$, the size of the state space is drastically reduced from 43 quintillion to merely 21, making the estimation of $\mathrm{TV}(Y_m,Y_\infty)$ feasible. Since computing $Y=d_o(X)$ requires finding the shortest solution for $X$, known as God’s algorithm\footnotemark, sampling $Y$ is computationally intensive. Fortunately, Big Red 200 is powerful enough to generate millions of samples per day. 
    With 1 million samples of $Y_m$, 1 million samples of $Y_\infty$, and 1000 bootstrap resamples, we have
    \[
    \mathrm{TV}(X_m,X_\infty)\geq \mathrm{TV}(Y_m,Y_\infty)=0.286329\pm0.000697>0.25.
    \]
    \footnotetext{\url{https://github.com/rokicki/cube20src}}
\end{proof}
\label{main}
\section{``Trilateration''}
\label{tri}
To fully leverage the above distance-based approach, we create the {\it Trilateration Dataset} as follows. We generate 1 million samples for each of $X_1,...,X_{52}$, as well as for $X_\infty$, resulting in a total of 53 million random cubes. For each of them, we compute the distance to the origin $d_o$ (Figure \ref{pattern}, Left), the distance to the superflip $d_s$ (Figure \ref{pattern}, Middle), and the distance to the checkerboard $d_c$ (Figure \ref{pattern}, Right). 
With this dataset at hand, we can analyze the convergence of $d_o(X)$, $d_s(X)$, and $d_c(X)$ to gain deeper insights into the convergence of $X$.\\
\begin{figure}[ht]
\label{pattern}
    \centering
    \begin{minipage}{0.33\textwidth}
        \centering
        \includegraphics[width=0.6\linewidth]{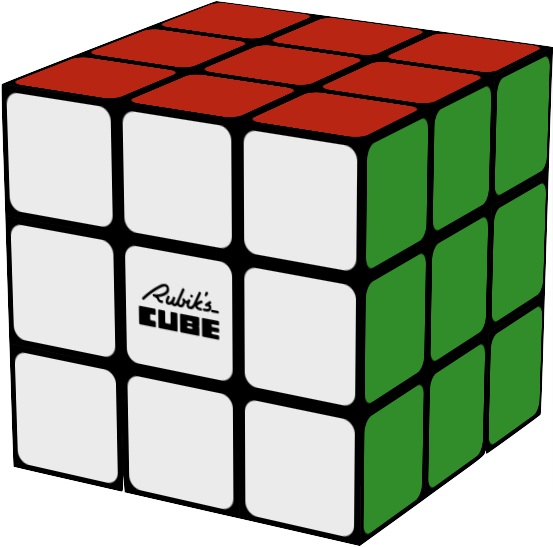}
    \end{minipage}%
    \begin{minipage}{0.33\textwidth}
        \centering
        \includegraphics[width=0.6\linewidth]{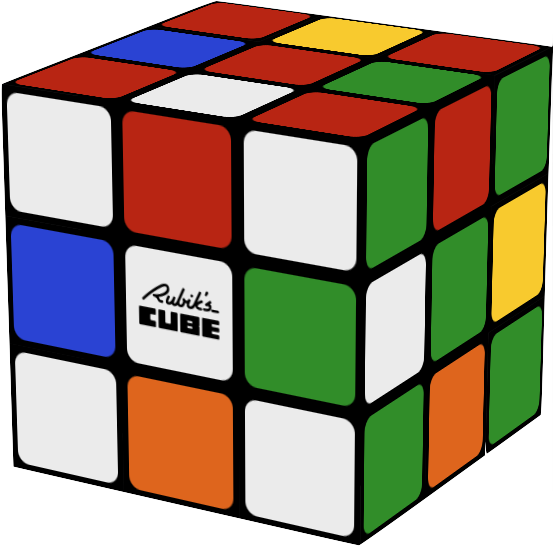}
    \end{minipage}%
    \begin{minipage}{0.33\textwidth}
        \centering
        \includegraphics[width=0.6\linewidth]{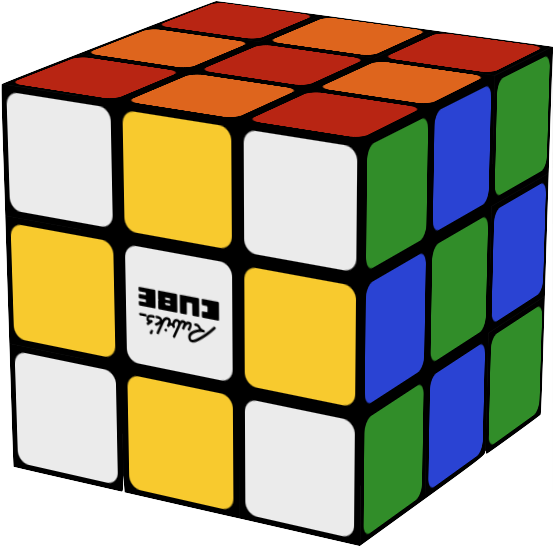}
    \end{minipage}
    \caption{Left: the origin $o$. Middle: the superflip $s$. Right: the checkerboard $c$.}
\end{figure}

To visualize the convergence, we plot not only the TV decay curve but also a sequence of histograms to illustrate how the marginal distribution approaches its limit. 
These sample-based plots can effectively be viewed as their population-based counterparts, because 1 million samples are more than enough to eliminate any visually detectable error when estimating a distribution with support size 21, as illustrated by the bootstrap in the above proof.

It turns out that combining $d_o(X)$, $d_s(X)$, and $d_c(X)$ does not yield an improved lower bound, so we analyze the convergence of each distance separately in the following. Improving the lower bound via combining $d_o(X)$ with some other functionals of $X$ is left for future research.

\subsection{Distance to the origin}

\begin{figure*}[ht]
\begin{center}
\centerline{\includegraphics[width=1.1\linewidth]{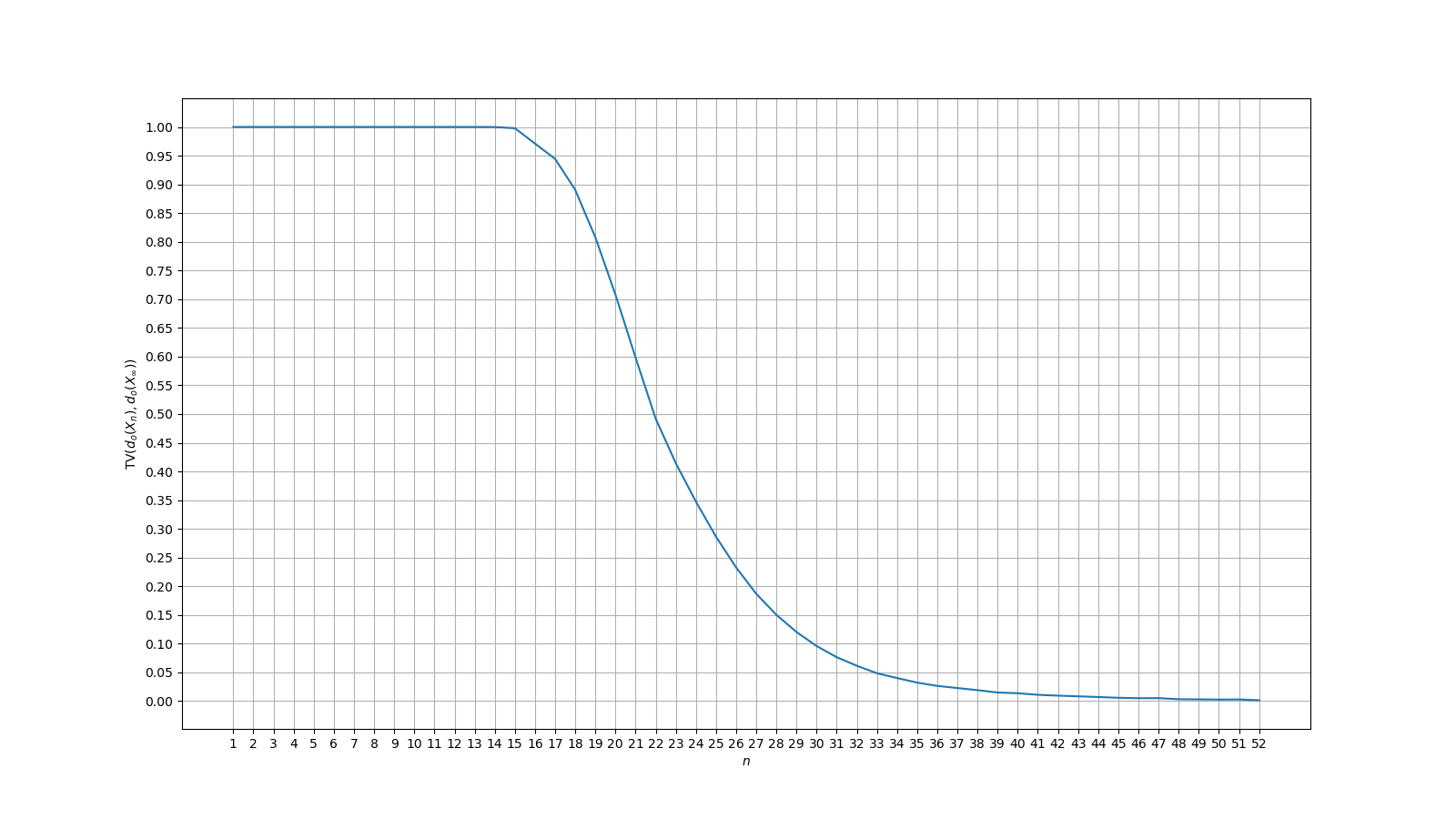}}
\caption{The decay curve of $\mathrm{TV}(d_o(X_n),d_o(X_\infty))$.}
\label{decay_o}
\end{center}
\end{figure*}

The TV decay curve in Figure \ref{decay_o} provides a clear overview of how $d_o(X_n)$ converges to $d_o(X_\infty)$, offering a lower bound for each $t_\mathrm{mix}(\epsilon)$. For example,
\[
t_\mathrm{mix}(0.5)\geq22,\;\;t_\mathrm{mix}(0.4)\geq24,\;\;t_\mathrm{mix}(0.3)\geq25,\;\;t_\mathrm{mix}(0.2)\geq27,\;\;t_\mathrm{mix}(0.1)\geq30,
\]
where the mixing time of $X$ is lower bounded by the mixing time of $d_o(X)$.
This functional of $X$ (distance to the origin) has a small state space (21 possible values) yet converges slowly (mixing time 26). These two properties are crucial for establishing a computable non-trivial lower bound on the convergence of $X$. However, human-made functionals (constructed without God’s algorithm) lack these properties. For example, the embedded 2x2x2 cube is a functional of the 3x3x3 cube, obtained by ignoring all edge pieces. Despite having millions of possible configurations, 19 is its mixing time\footnotemark. This comparison naturally raises the question of why the convergence of $d_o(X)$ is so slow. 

In Figure \ref{drift_o}, we plot the distributions of $d_o(X_{10})$, $d_o(X_{20})$, $d_o(X_{30})$, and $d_o(X_\infty)$. Starting from a point mass at 0, the convergence is all about moving towards 17 and 18, which together account for over 93\% of the mass in the limiting distribution. Roughly speaking, the mixing time of $d_o(X)$ is the time when 70\% of the mass is moved to 17 and 18. Instead of ``mixing'', ``moving'' (from 0 all the way to 17 and 18) better describes what happens here, which explains why the convergence of $d_o(X)$ is so slow.
Back to $\X$, since the scrambling process is highly symmetric, the mass should be moved from the origin to the ``outer shell'' in a uniform manner. If this is true, then $t_\mathrm{mix}\geq26$ may be tight.

\footnotetext{\url{https://theconversation.com/how-hard-is-it-to-scramble-rubiks-cube-129916}}

\begin{figure}[ht]
    \centering
    \begin{minipage}{0.25\textwidth}
        \centering
        \includegraphics[width=1\linewidth]{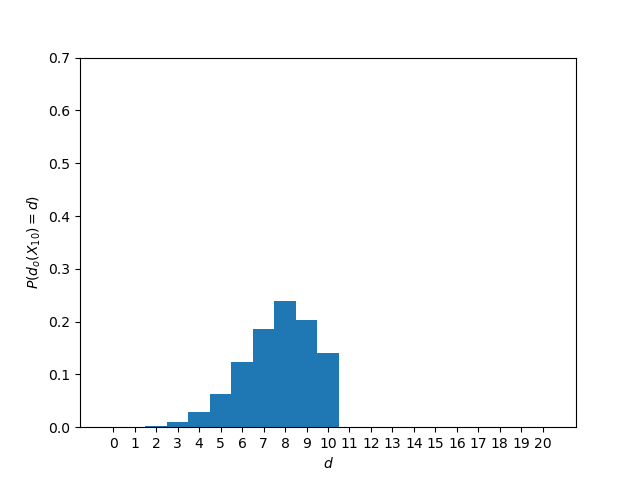}
    \end{minipage}%
    \begin{minipage}{0.25\textwidth}
        \centering
        \includegraphics[width=1\linewidth]{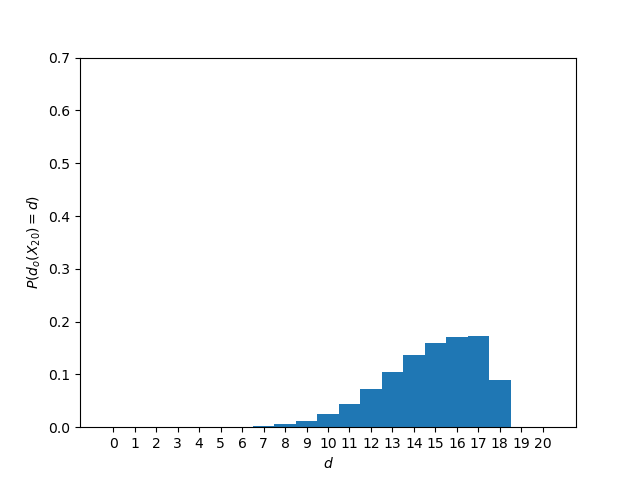}
    \end{minipage}%
    \begin{minipage}{0.25\textwidth}
        \centering
        \includegraphics[width=1\linewidth]{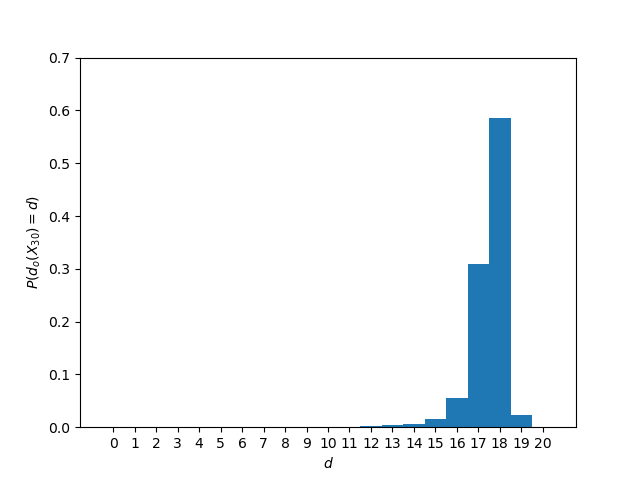}
    \end{minipage}%
    \begin{minipage}{0.25\textwidth}
        \centering
        \includegraphics[width=1\linewidth]{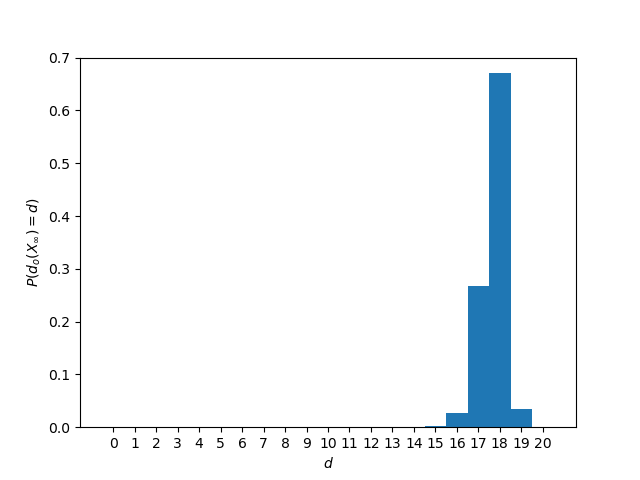}
    \end{minipage}
    \caption{The distributions of $d_o(X_{10})$, $d_o(X_{20})$, $d_o(X_{30})$, and $d_o(X_\infty)$.}
    \label{drift_o}
\end{figure}

\subsection{Distance to the superflip}

\begin{figure*}[ht]
\begin{center}
\centerline{\includegraphics[width=1.1\linewidth]{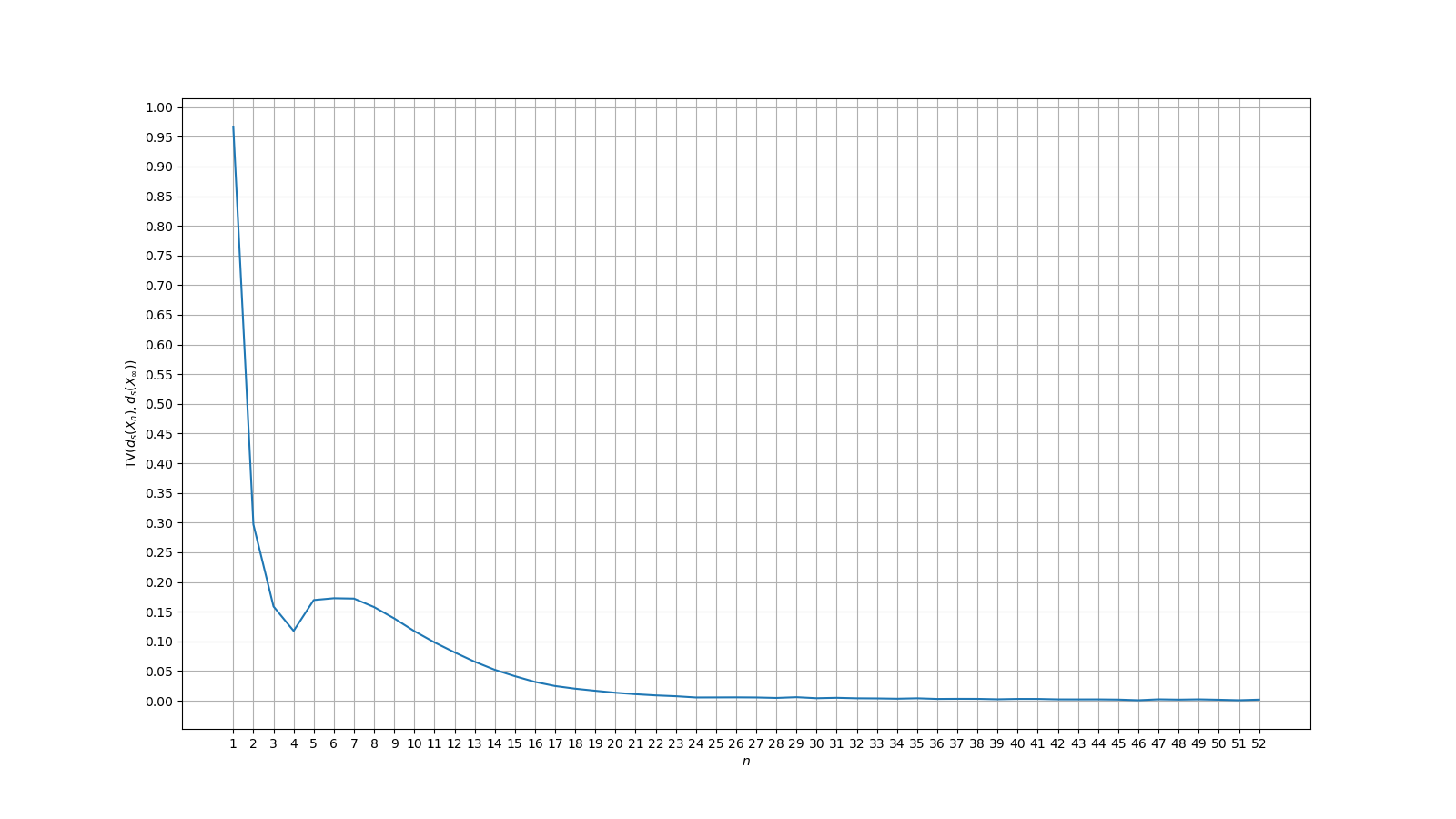}}
\caption{The decay curve of $\mathrm{TV}(d_s(X_n),d_s(X_\infty))$.}
\label{decay_s}
\end{center}
\end{figure*}

The superflip requires 20 moves to solve\footnotemark, so $d_s(X)$ starts from 20 instead of 0. As shown in Figure \ref{decay_s}, $d_s(X)$ converges almost instantly (mixing time 3). Recall that the mixing time of $d_o(X)$ is 26. This comparison naturally raises the question of why the convergence is $d_s(X)$ is so fast.

In Figure \ref{drift_s}, note that the distribution of $d_s(X_\infty)$ is the same as $d_o(X_\infty)$, concentrating on 17 and 18. This is because $X_\infty$ follows the uniform distribution on $\X$ and any configuration can be defined as the ``origin'' of $\X$, i.e., for each $d\in\{0,...,20\}$, the sets $\{x:d_s(x)=d\}$ and $\{x:d_o(x)=d\}$ have the same number of elements. Although $d_o(X)$ and $d_s(X)$ share the same limit, the former starts from 0, while the latter starts from 20. Starting from 20, 3 random moves are enough to move 70\% of the mass to 17 and 18, hence the mixing time. It is interesting to note that the TV decay curve in Figure \ref{decay_s} is not monotonic. Specifically, 
\[\mathrm{TV}(d_s(X_5),d_s(X_\infty))>\mathrm{TV}(d_s(X_4),d_s(X_\infty)).\] As shown in Figure \ref{drift_s}, some mass is moved from 18 ``back'' to 17. This is possible because $X_5$ can reach at most $18^5=1,889,568$ configurations, a negligibly small fraction of $\X$ that cannot represent $\X$ in terms of the distribution of the distance to the superflip.
Although $d_s(X)$ converges too fast to generate any non-trivial lower bound on the convergence of $X$, it serves as a non-trivial example to illustrate that the TV decay curve of a hidden Markov chain may not always be decreasing.

\footnotetext{\url{http://www.math.rwth-aachen.de/~Martin.Schoenert/Cube-Lovers/michael_reid__superflip_requires_20_face_turns.html}}

\begin{figure}[ht]
    \centering
    \begin{minipage}{0.25\textwidth}
        \centering
        \includegraphics[width=1\linewidth]{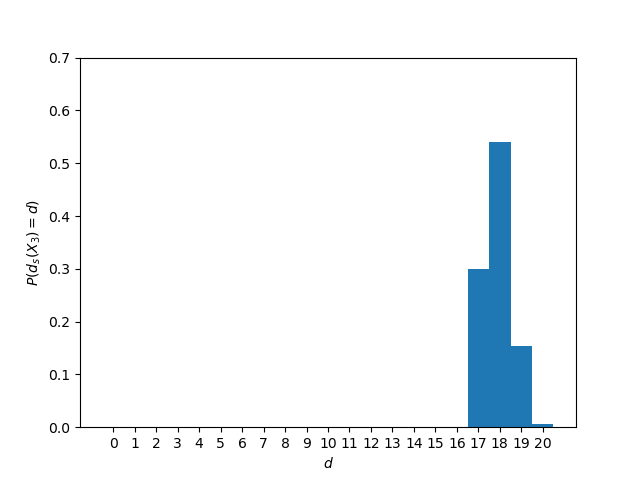}
    \end{minipage}%
    \begin{minipage}{0.25\textwidth}
        \centering
        \includegraphics[width=1\linewidth]{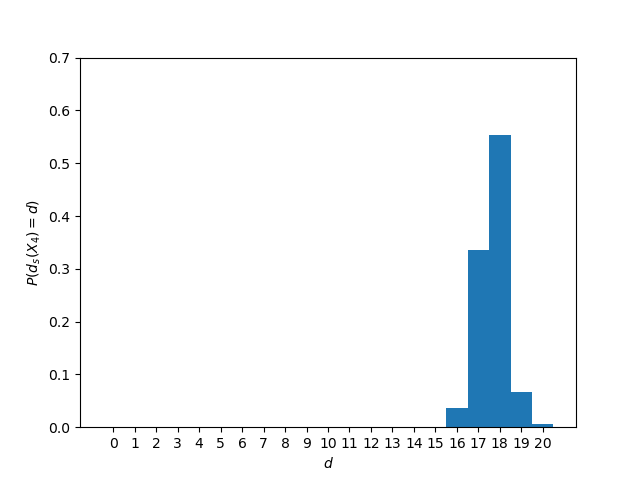}
    \end{minipage}%
    \begin{minipage}{0.25\textwidth}
        \centering
        \includegraphics[width=1\linewidth]{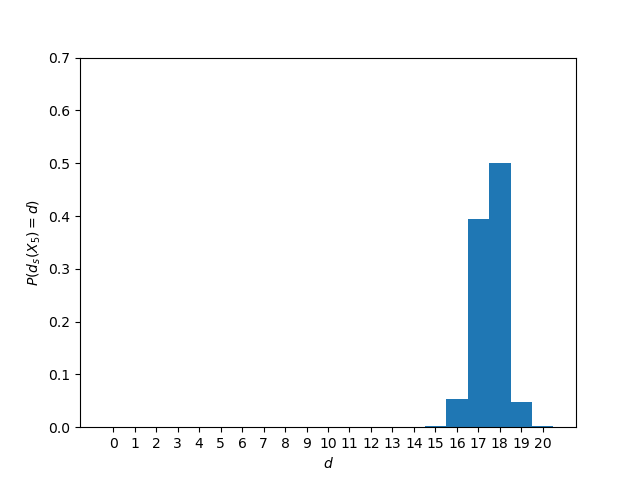}
    \end{minipage}%
    \begin{minipage}{0.25\textwidth}
        \centering
        \includegraphics[width=1\linewidth]{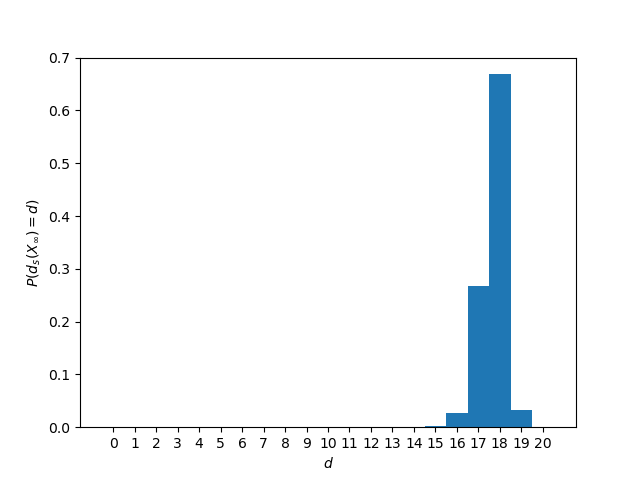}
    \end{minipage}
    \caption{The distributions of $d_s(X_{3})$, $d_s(X_{4})$, $d_s(X_{5})$, and $d_s(X_\infty)$.}
    \label{drift_s}
\end{figure}

\subsection{Distance to the checkerboard}

\begin{figure*}[ht]
\begin{center}
\centerline{\includegraphics[width=1.1\linewidth]{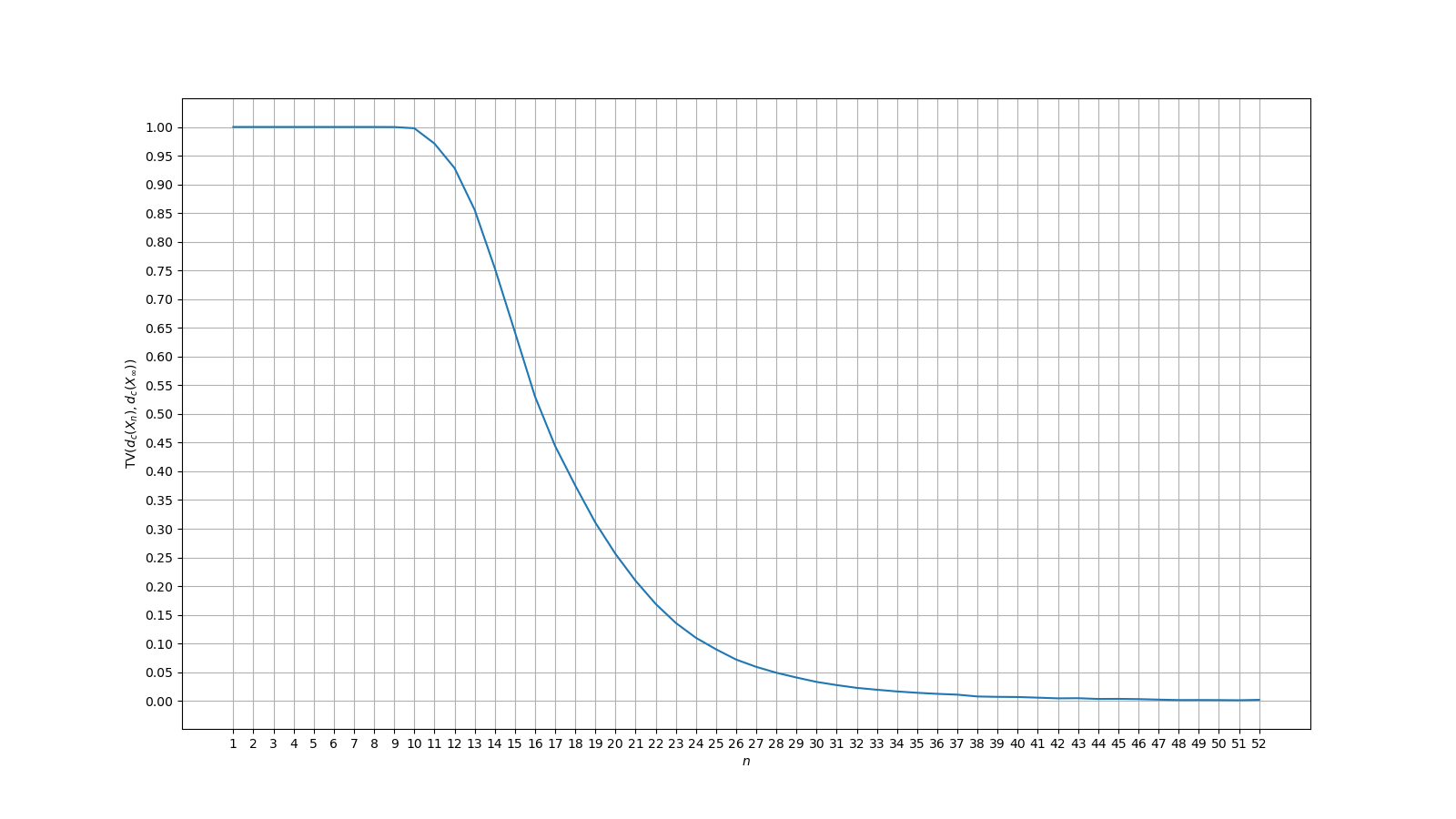}}
\caption{The decay curve of $\mathrm{TV}(d_c(X_n),d_c(X_\infty))$.}
\label{decay_c}
\end{center}
\end{figure*}

The checkerboard clearly requires 6 moves to solve (e.g., U2D2F2B2L2R2), so $d_c(X)$ starts from $6$. Similar to $d_o(X)$, the convergence of $d_c(X)$ is all about moving towards 17 and 18, from left to right. Since it starts closer to 17 and 18, its convergence is faster, as shown in Figure \ref{decay_c} and \ref{drift_c}.

\begin{figure}[ht]
    \centering
    \begin{minipage}{0.25\textwidth}
        \centering
        \includegraphics[width=1\linewidth]{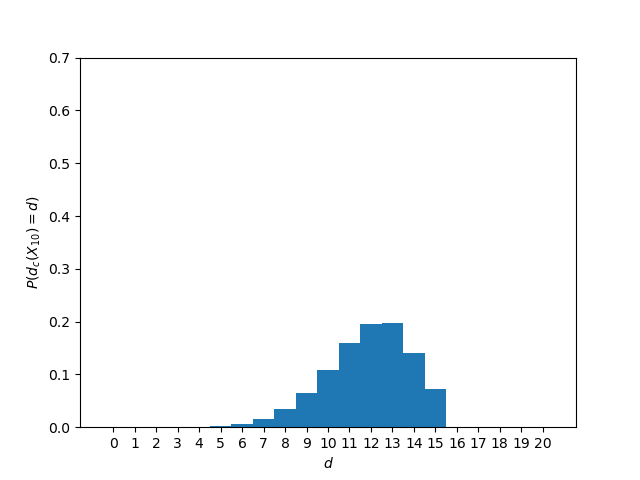}
    \end{minipage}%
    \begin{minipage}{0.25\textwidth}
        \centering
        \includegraphics[width=1\linewidth]{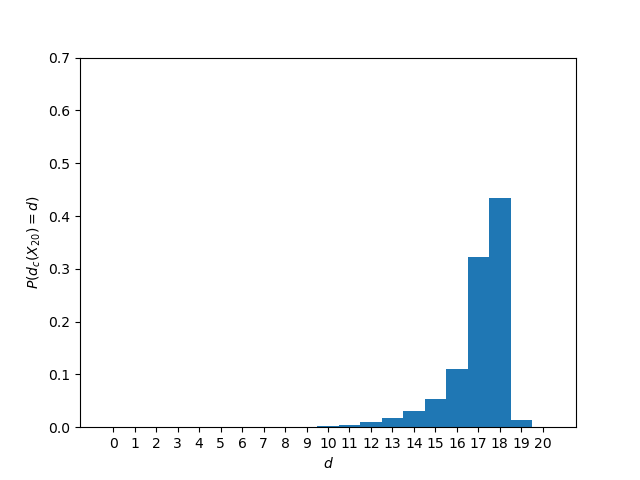}
    \end{minipage}%
    \begin{minipage}{0.25\textwidth}
        \centering
        \includegraphics[width=1\linewidth]{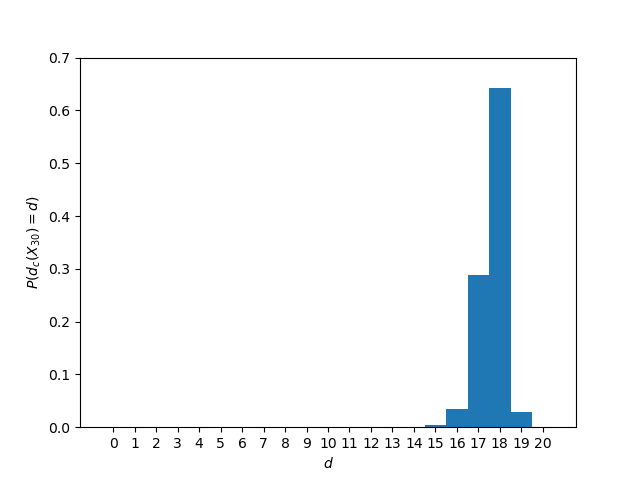}
    \end{minipage}%
    \begin{minipage}{0.25\textwidth}
        \centering
        \includegraphics[width=1\linewidth]{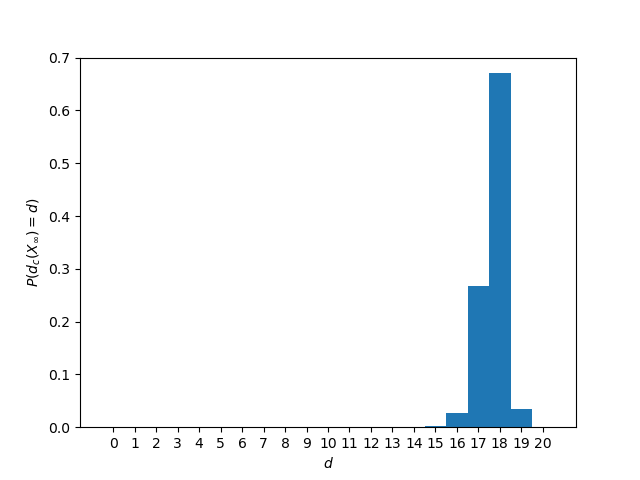}
    \end{minipage}
    \caption{The distributions of $d_c(X_{10})$, $d_c(X_{20})$, $d_c(X_{30})$, and $d_c(X_\infty)$.}
    \label{drift_c}
\end{figure}

\begin{ack}
    The third author thanks Lilly Endowment Inc. for its support through the Indiana University Pervasive Technology Institute.
\end{ack}

\bibliographystyle{plainnat}
\bibliography{references}
\end{document}